\documentclass[twoside]{amsart}

\numberwithin{equation}{section}
\usepackage{amsmath,amssymb,amsfonts,amsthm,latexsym}
\usepackage{graphicx,color,enumerate}
\usepackage[margin=1in]{geometry}
\usepackage{longtable}
\usepackage{appendix}
\usepackage{xcolor}
\usepackage{stmaryrd}
\usepackage{mleftright}
\usepackage{breqn}
\usepackage{url}
\usepackage{orcidlink}

\newtheorem{theorem}{Theorem}[section]

\newtheorem{proposition}[theorem]{Proposition}
\newtheorem{corollary}[theorem]{Corollary}

\theoremstyle{definition}

\theoremstyle{remark}
\newtheorem{remark}[theorem]{Remark}

\begin{document}

\title{Factorization of spread polynomials}

\author[J. Cigler]{
Johann Cigler\,\orcidlink{0000-0002-5203-3977
}}
\address{Department of Mathematics,  Oskar-Morgenstern-Platz 1, 1090 Vienna, Austria}
\email{johann.cigler@univie.ac.at}
\author[H.-C.~Herbig]{Hans-Christian Herbig\,\orcidlink{0000-0003-2676-3340}}
\address{Departamento de Matem\'{a}tica Aplicada, Universidade Federal do Rio de Janeiro,
Av. Athos da Silveira Ramos 149, Centro de Tecnologia - Bloco C, CEP: 21941-909 - Rio de Janeiro, Brazil}
\email{herbighc@gmail.com}


\begin{abstract} 
We present a proof of a conjecture of Goh and Wildberger on the factorization of the spread polynomials.
We indicate how the factors can be effectively calculated and exhibit a connection to the factorization of Fibonacci numbers into primitive parts. 
\end{abstract}

\maketitle
\tableofcontents
\section{Introduction}\label{sec:intro}

The sequence of \textit{spread polynomials} $( S_{n}( x))_{n\geq 1}$ of Norman J. Wildberger is uniquely defined by the requirement
\begin{align}
S_{n}\left(\sin^{2} \theta \right) =\sin^{2}( n\theta ) .
\end{align}
The spread polynomials play a central role in Wildberger's \textit{rational trigonometry} \cite{divine}. In order to make plane geometry free of transcendental expressions and square roots he suggests to replace lengths by their squares, which are referred to as \textit{quadrances}, and angles by \textit{spreads}.
In the right angled triangle with sides $ a,b$ and $ c$
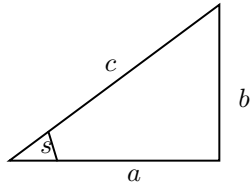
\begin{center}
\tikzset{every picture/.style={line width=0.75pt}} 

\begin{tikzpicture}[x=0.5pt,y=0.5pt,yscale=-1,xscale=1]

\draw   (174.79,8.94) -- (16,127) -- (174.79,127) -- cycle ;
\draw    (45.33,104.58) -- (52.27,127.52) ;

\draw (102.76,131.21) node [anchor=north west][inner sep=0.75pt]    {$a$};
\draw (187.01,70.07) node [anchor=north west][inner sep=0.75pt]    {$b$};
\draw (85.9,47.9) node [anchor=north west][inner sep=0.75pt]    {$c$};
\draw (36.9,110.67) node [anchor=north west][inner sep=0.75pt]    {$s$};
\end{tikzpicture}
\end{center}
the spread $ s\in [ 0,1]$ is defined as the ratio $ s=b^{2} /c^{2}$ of the quadrance of the opposite leg $b$ by the quadrance of the hypothenuse $ c$. The geometric property of the spread polynomials that makes them play a central role in rational trigonometry is the relation 
\begin{align}
    s_{n} =S_{n}( s)
    \end{align}
for an arrangement of $ n+1$ lines which all meet in a single point and whose spreads of neighboring lines all coincide. Here $ s_{n}$ is the spread between the extremal lines of the arrangement. 
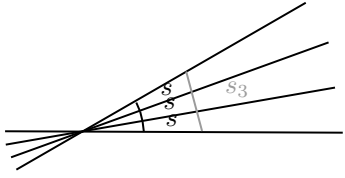
\begin{center}
\tikzset{every picture/.style={line width=0.75pt}} 

\begin{tikzpicture}[x=0.5pt,y=0.5pt,yscale=-1,xscale=1]

\draw    (12.65,106.85) -- (268.12,107.92) ;
\draw    (13.18,116.98) -- (262.25,73.78) ;
\draw    (17.18,126.58) -- (257.38,39.58) ;
\draw    (21.18,135.91) -- (240.12,9.55) ;
\draw    (117.67,107.42) -- (116.67,99.42) ;
\draw    (116.67,99.42) -- (114.87,92.04) ;
\draw    (111.93,85.11) -- (114.87,92.04) ;
\draw [color={rgb, 255:red, 155; green, 155; blue, 155 }  ,draw opacity=1 ]   (161.67,106.65) -- (149.67,61.59) ;

\draw (131.67,91.71) node [anchor=north west][inner sep=0.75pt]    {$s$};
\draw (130.33,79.44) node [anchor=north west][inner sep=0.75pt]    {$s$};
\draw (127.93,68.51) node [anchor=north west][inner sep=0.75pt]    {$s$};
\draw (176.92,66.91) node [anchor=north west][inner sep=0.75pt]  [color={rgb, 255:red, 155; green, 155; blue, 155 }  ,opacity=1 ]  {$s_{3}$};

\end{tikzpicture}
\end{center}

The spread polynomials can be written in terms of the Lucas polynomials $L_{n}( x)$
\begin{align}
S_{n}( x) =\frac{2-L_{n}( 2-4x)}{4} .
\end{align}
Following \cite{numerology} we consider the closely related \textit{zpread polynomials}
\begin{align}\label{eq:Zdef}
    Z_{n}( x) =4S_{n}\left(\frac{x}{4}\right) =2-L_{n}( 2-x) .
\end{align}    
A conjecture of Wildberger's honours student Shuxiang Goh (see \cite{Wildegg}), stated in terms of $Z_{n}( x)$, claims the following.
\begin{theorem}\label{thm:GohWild}
There is a sequence of polynomials $( \Phi _{d}( x))_{d\geq 1}$, $\Phi_d(x)\in\mathbb{Z}[x]$, such that $Z_{n}( x) =\prod _{d|n} \Phi _{d}( x)$ and
$\deg( \Phi _{d}) =\varphi ( d)$, where $\varphi ( d)$ is Euler's totient function.
\end{theorem}

The aim of this note is to prove the theorem as a consequence of \eqref{eq:Zdef} and known results about the Chebyshev
polynomials obtained in \cite{Grubb}, \cite{Guertas} and \cite{WZ}. Moreover, we obtain some concrete facts about the polynomials $\Phi _{d}( x)$, some of which had already been conjectured in \cite{numerology}.

\vspace{2mm}

\noindent \emph{Acknowledgements.} We would like to thank Tri Nguyen for generously sharing the idea \cite{TriNguyen} with us.

\section{Some well-known facts about Lucas polynomials}\label{sec:Lucas}

The Lucas polynomials $L_{n}( x)$ are defined by the recursion 
\begin{align}
L_{n}( x) =xL_{n-1}( x) -L_{n-2}( x)
\end{align}
for $n\geq 2$ with initial values $L_{0}( x) =2$ and $L_{1}( x) =x$.\footnote{Up to the index $n=0$ the sequence of Lucas polynomials
coincides with the sequence of the pyramidal polynomials from \cite{numerology}, the latter coming from a Riordan array.}
The first terms are (cf. \cite[entry A034807]{OEIS}):
\begin{align*}
\begin{array}{ c|c c c c c c c c }
n & 0 & 1 & 2 & 3 & 4 & 5 & 6 & \cdots \\
\hline
L_{n}( x) & 2 & x & -2+x^{2} & -3x+x^{3} & 2-4x^{2} +x^{4} & 5x-5x^{3} +x^{5} & -2+9x^{2} -6x^{4} +x^{6} & \cdots 
\end{array}.
\end{align*}
Binet's formula gives
\begin{align}
L_{n}( x) =( \alpha ( x))^{n} +( \beta ( x))^{n}
\end{align}
with $\alpha ( x) =\frac{x+\sqrt{x^{2} -4}}{2}$ and $\beta ( x) =\frac{1}{\alpha ( x)} =\frac{x-\sqrt{x^{2} -4}}{2}$.
The Lucas polynomials are related to the Chebyshev polynomials of the first kind $T_{n}( x)$ by
\begin{align}
L_{n}( x) =2T_{n}\left(\frac{x}{2}\right)
\end{align}
and are characterized by 
\begin{align}\label{eq:LucTrig}
L_{n}\left( z+\frac{1}{z}\right) =z^{n} +\frac{1}{z^{n}}
\end{align}
for $n\geq 1$ and \ $z\neq 0$ because $z^{n} +\frac{1}{z^{n}} =\left( z +\frac{1}{z}\right)\left( z^{n-1} +\frac{1}{z^{n-1}}\right) -\left( z^{n-2} +\frac{1}{z^{n-2}}\right)$.
For $z=e^{\sqrt{-1} \theta }$ Equation \eqref{eq:LucTrig} implies 
\begin{align}L_{n}( 2\cos \theta ) =2\cos( n\theta ).
\end{align}
Equation \eqref{eq:LucTrig} also implies 
\begin{align}\label{eq:Lmultiplicative}
L_{mn}( x) =L_{m}( L_{n}( x))
\end{align}
for all $m,n\geq 0$ since $z^{mn} +\frac{1}{z^{mn}} =\left( z^{n}\right)^{m} +\left(\frac{1}{z^{n}}\right)^{m}$.

Let us also mention that $ z^{2n} +\frac{1}{z^{2n}} -2=\left( z^{n} +\frac{1}{z^{n}} -2\right)\left( z^{n} +\frac{1}{z^{n}} +2\right)$ implies
\begin{align} \label{eq:L1stBinomi}
L_{2n}( x) -2=( L_{n}( x) -2)( L_{n}( x) +2)
\end{align}
and that
\begin{align}\label{eq:Levensq}
L_{2n}( x) +2=( L_{n}( x))^{2}
\end{align}
because $ z^{2n} +\frac{1}{z^{2n}} +2=\left( z^{n} +\frac{1}{z^{n}}\right)^{2}$.

\section{Factorization of the polynomials $L_{n}( x) -2$}\label{sec:LucFact}
\begin{proposition}
For $n\geq 1$ the roots of $f_n(x):=L_{n}( x) -2$ are $\nu _{k} =2\cos\left(\frac{2\pi k}{n}\right)$ with $0\leq k\leq \frac{n}{2}$.
For odd $n=2m+1$ the roots $\nu _{k}$ with $1\leq k\leq m$ are double roots and $\nu _{0} =2$ is a simple root.
For even $n=2m$ the roots $\nu _{k}$ with $1\leq k< m$ are double roots and $\nu _{0} ,\nu _{m}$ are  simple roots.
\end{proposition}
\begin{proof}
Since $f_{n}\left( 2\cos\left(\frac{2\pi k}{n}\right)\right) =L_{n}\left( 2\cos\left(\frac{2\pi k}{n}\right)\right) -2=2\cos( 2\pi k) -2=0$ each $\nu _{k} =2\cos\left(\frac{2\pi k}{n}\right)$ for $k\in \mathbb{Z}$ is a root of $f_{n}( x)$. All different roots of this form are given by $0\leq k\leq \frac{n}{2}$. 

We claim that in this way we obtain all the roots of $f_{n}( x)$ because the number of roots above, counted with multiplicity, in fact already amounts to $n$.

To show the claim for $n=2m+1$ it suffices to show that 
\begin{align}\label{eq:evenSqId}
( L_{2m+1}( x) -2)( x-2) =( L_{m+1}( x) -L_{m}( x))^{2},
\end{align}
which is a square of a polynomial. Identity \eqref{eq:evenSqId} in turn can be shown by setting $x=z+\frac{1}{z}$ and evaluating
\begin{align*}
\left( z^{2m+1} +\frac{1}{z^{2m+1}} -2\right)\left( z +\frac{1}{z} -2\right)& =\frac{\left( z^{2( 2m+1)} -2z^{2m+1} +1\right)\left( z^{2} -2z+1\right)}{z^{2m+2}} \\&=\left(\frac{\left( z^{2m+1} -1\right)( z -1)}{z^{m+1}}\right)^{2} =\left(\left( z^{m+1} +\frac{1}{z^{m+1}}\right) -\left( z^{m} +\frac{1}{z^{m}}\right)\right)^{2}
\end{align*}
For $n=2m$ in turn it suffices to show the identity 
\begin{align}
    ( L_{2m}( x) -2)( x-2)( x+2) =( L_{m+1}( x) -L_{m-1}( x))^{2} .
    \end{align}
This, however, follows from
$$\frac{z^{2m} +\frac{1}{z^{2m}} -2}{\left( z +\frac{1}{z} -2\right)\left( z +\frac{1}{z} +2\right)} =\left(\frac{\left( z^{2m} -1\right)}{z^{m-1}\left( z^{2} -1\right)}\right)^{2}$$ 
and $\left( z +\frac{1}{z} -2\right)\left( z +\frac{1}{z} +2\right) =\left(\frac{\left( z^{2} -1\right)}{z}\right)^{2}$ because 
\begin{align*}
\left( z^{2m} +\frac{1}{z^{2m}} -2\right)\left(\frac{\left( z^{2} -1\right)}{z}\right)^{2} &=\left(\frac{\left( z^{2m} -1\right)}{z^{m-1}\left( z^{2} -1\right)}\left(\frac{\left( z^{2} -1\right)}{z}\right)^{2}\right)^{2}=\left(\frac{\left( z^{2m} -1\right)\left( z^{2} -1\right)}{z^{m+1}}\right)^{2}
\\&=\left( z^{m+1} -\frac{1}{z^{m-1}} -z^{m-1} +\frac{1}{z^{m+1}}\right)^{2} .
\end{align*}
\end{proof}

Next we use the fact that a polynomial $p( x)$ with an algebraic number $a$ as a root has the minimal polynomial of $a$ as a factor.    

The minimal polynomial $ \psi _{n}( x)$ of 2$ \cos\left(\frac{2\pi }{n}\right)$ is
\begin{align}\label{eq:psin}
\psi _{n}( x) \ =\prod _{\gcd( j,n) =1,\ 0< j< n/2}\left( x-2\cos\frac{2\pi j}{n}\right)
\end{align}
for $ n >2$ with $ \psi _{1}( x) =x-2$ and $ \psi _{2}( x) =x+2$.
The first examples are:
\begin{align*} \begin{array}{ c|c c c c c c c c c c }
n & 1 & 2 & 3 & 4 & 5 & 6 & 7 & 8 & 9 & \cdots \\
\hline
\psi _{n}( x) & -2+x & 2+x & 1+x & x & -1+x+x^{2} & -1+x & -1-2x+x^{2} +x^{3} & -2+x^{2} & 1-3x+x^{3} & \cdots 
\end{array}.
\end{align*}
\begin{theorem}
Let $\alpha ( x) =\frac{x+\sqrt{x^{2} -4}}{2}$ and $C_{n}( x)$ be the $n$th cyclotomic polynomial. Then for $n\geq 3$ the minimal 
polynomial of $2\cos\left(\frac{2\pi }{n}\right)$ is
\begin{align}
\psi _{n}( x) =\frac{C_{n}( \alpha ( x))}{( \alpha ( x))^{\varphi ( n) /2}} .
\end{align}
\end{theorem}

\begin{proof}
For $n\geq 3$ the cyclotomic polynomial $C_{n}( x)$ is a symmetric polynomial with integer coefficients of even degree
$\varphi ( n)$. For example:
\begin{align*}
\begin{array}{ c|c c c c c c c c }
n & 3 & 4 & 5 & 6 & 7 & 8 & 9 & \cdots \\
\hline
C_{n}( x) & 1+x+x^{2} & 1+x^{2} & 1+x+x^{2} +x^{3} +x^{4} & 1-x+x^{2} & 1+x+\cdots +x^{6} & 1+x^{4} & 1+x^{3} +x^{6} & \cdots 
\end{array}.
\end{align*}
Therefore $\frac{C_{n}(x)}{x^{\varphi ( n) /2}}$ is a sum of terms of the form $x^{k} +\frac{1}{x^{k}}$ with integer coefficients. Since $L_{k} =\alpha ^{k} +\frac{1}{\alpha ^{k}}$
we can write
$$\frac{C_{n}( \alpha ( x))}{( \alpha ( x))^{\varphi ( n) /2}} =\sum _{k=0}^{\varphi ( n) /2} c_{k} L_{k}( x) \in \mathbb{Z}[ x]$$  with integer coefficients $c_{k}$ and note that this polynomial is of degree $\varphi ( n) /2$. Since $\alpha ( 2\cos \theta ) =\cos \theta +\sqrt{\cos^{2} \theta -1} =\cos \theta +\sqrt{-\sin^{2} \theta } =e^{\sqrt{-1} \theta }$ we get 
$$C_{n}\left( \alpha \left( 2\cos\frac{2\pi }{n}\right)\right) =C_{n}\left( e^{2\sqrt{-1} \pi /n}\right) =0.$$
Therefore $\frac{C_{n}( \alpha ( x))}{( \alpha ( x))^{\varphi ( n) /2}}$ is a monic polynomial of degree $\varphi ( n) /2$ with integer coefficients and root $2\cos\frac{2\pi }{n}$. Consequently it must coincide with the minimal polynomial $\psi _{n}( x)$.
\end{proof}

For example, $ \frac{C_{9}( x)}{x^{3}} =1+x^{3} +\frac{1}{x^{3}}$ gives  $ \frac{C_{9}( \alpha ( x))}{( \alpha ( x))^{3}} =1+( \alpha ( x))^{3} +\frac{1}{( \alpha ( x))^{3}} =1+L_{3}( x) =1-3x+x^{3} =\psi _{9}( x)$.
\begin{corollary}
For $ n\geq 1$ we have $ \psi _{2^{n+2}}( x) =L_{2^{n}}( x)$.
\end{corollary}
\begin{proof}Since $ C_{2^{n+2}}( x) =x^{2^{n+1}} +1$ we get $ \frac{C_{2^{n+2}}( x)}{x^{2^{n}}} =x^{2^{n}} +\frac{1}{x^{2^{n}}}$ and $ ( \alpha ( x))^{2^{n}} +\frac{1}{( \alpha ( x))^{2^{n}}} =L_{2^{n}}( x) .$
\end{proof}

\begin{proposition}{\cite[Proposition 2.2]{Grubb}} \label{prop:Grubb} For $n\geq 1$ we have
$ L_{n}( x) -2=\psi _{1}( x) \left(\psi _{2}(x)\right)^{e_{n}}\prod _{k|n,k\neq 1,2} \psi _{k}^{2}( x)$
with $ e_{n} =\left( 1+( -1)^{n}\right) /2$.
\end{proposition}
\begin{proof}
    This follows from the multiplicities of the roots of $ L_{n}( x) -2$ and the fact that for $ \gcd( j,n) =d$ the minimal polynomial of $ 2\cos( 2\pi j/n)$ is $ \psi _{n/d}( x)$.
\end{proof}

\section{Some properties of $ Z_{n}( x)$}
The formula $ Z_{n}( x) =2-L_{n}( 2-x)$ shows that $ Z_{n}( x)$ is closely related to the Lucas polynomials $ L_{n}( x)$. The first terms are:
$$
\begin{array}{ c|c c c c c c }
n & 1 & 2 & 3 & 4 & 5 & \cdots \\
\hline
Z_{n}( x) & x & 4x-x^{2} & 9x-6x^{2} +x^{3} & 16x-20x^{2} +8x^{3} -x^{4} & 25x-50x^{2} +35x^{3} -10x^{4} +x^{5} & \cdots 
\end{array}.$$
Note that $ -Z_{n}( -x)$ is a monic polynomial of degree $ n$ with positive coefficients. Since $ L_{n}( 2) =2$ we have $ Z_{n}( 0) =0$.

Let $ c( n,k) =\left[ x^{k}\right] L_{n}( 2+x)$. From $ L_{n}( 2+x) =( 2+x) L_{n-1}( 2+x) -L_{n-2}( 2+x)$ we get
 $ c( n,k) =2c( n-1,k) +c( n-1,k-1) -c( n-2,k)$. Since $ c( 0,1) =0$ and $ c( 1,1) =1$ we get with induction that
$ c( n,1) =n^{2}$. It is easily verfied that $ c( n,k) =\binom{n+k-1}{n-k}\frac{n}{k}$. Therefore, we get
\begin{align}
Z_{n}( x) =\sum _{k=1}^{n}( -1)^{k-1}\binom{n+k-1}{n-k}\frac{n}{k} x^{k} .
\end{align}

With $ \lambda ( x) =\alpha ( 2-x) =\frac{2-x+\sqrt{x^{2} -4x}}{2}$ we get
\begin{align} 
Z_{n}( x) =\frac{\left( \lambda ( x)^{n} -1\right)^{2}}{\lambda ( x)^{n}}
\end{align}
because $ Z_{n}( x) =2-L_{n}( 2-x) =2-\left(( \lambda ( x))^{n} +\frac{1}{( \lambda ( x))^{n}}\right) =-\frac{\left(( \lambda ( x))^{n} -1\right)^{2}}{( \lambda ( x))^{n}}.$

Another connection with the Lucas polynomials is the following.

\begin{remark} \label{rm:ZL}
\begin{align*}
Z_{2n+1}\left( x^{2}\right) =( L_{2n+1}( x))^{2} ,\\
Z_{2n}\left( x^{2}\right) =4-( L_{2n}( x))^{2} .
\end{align*}
\end{remark}

\begin{proof}
For odd $ m$ we have $ Z_{m}\left( x^{2}\right) =2-L_{m}\left( 2-x^{2}\right) =2-L_{m}( -L_{2}( x)) =2+L_{2m}( x) =( L_{m}( x))^{2}$
because of Equation \eqref{eq:Levensq}. The second identity follows from
\begin{align*}
Z_{2n}\left( x^{2}\right) &=2-L_{2n}\left( 2-x^{2}\right) =2-L_{2n}\left( 2-\left( z+\frac{1}{z}\right)^{2}\right) =2-L_{2n}\left( z^{2} +\frac{1}{z^{2}}\right) =2-\left( z^{4n} +\frac{1}{z^{4n}}\right)\\
&=4-\left( z^{2n} +\frac{1}{z^{2n}}\right)^{2} =4-( L_{2n}( x))^{2} .
\end{align*}
\end{proof}

\begin{proposition}
For arbitrary $ u\neq 0$ we have
\begin{align}
Z_{n}\left( -\left( u-\frac{1}{u}\right)^{2}\right)&=-\left( u^{n} -\frac{1}{u^{n}}\right)^{2} ,\\
Z_{mn}( x) &=Z_{m}( Z_{n}( x)) .
\end{align}
\end{proposition}

\begin{proof}
Note that $ Z_{n}\left( -\left( u-\frac{1}{u}\right)^{2}\right) =2-L_{n}\left( 2+\left( u-\frac{1}{u}\right)^{2}\right) =2-L_{n}\left( u^{2} +\frac{1}{u^{2}}\right) =2-\left( u^{2n} +\frac{1}{u^{2n}}\right) =-\left( u^{n} -\frac{1}{u^{n}}\right)^{2}$.
On the other hand, $ Z_{mn}( x) =2-L_{mn}( 2-x) =2-L_{m}( L_{n}( 2-x)) =2-L_{m}( 2-Z_{n}( x)) =Z_{m}( Z_{n}( x))$.
\end{proof}

\section{Proof of Theorem \ref{thm:GohWild}}

Let $ z_{n}( x) =( -1)^{n-1} Z_{n}( x) =( -1)^{n-1}( 2-L_{n}( 2-x))$ be the monic versions of the $ Z_{n}( x)$. The first terms are:
$$ \begin{array}{ c|c c c c c c }
n & 1 & 2 & 3 & 4 & 5 & \cdots \\
\hline
z_{n}( x) & x & -4x+x^{2} & 9x-6x^{2} +x^{3} & -16x+20x^{2} -8x^{3} +x^{4} & 25x-50x^{2} +35x^{3} -10x^{4} +x^{5} & \cdots 
\end{array}.$$
Proposition \ref{prop:Grubb} gives
\begin{align}
z_{n}( x) =\psi _{1}( 2-x)( \psi _{2}( 2-x))^{e_{n}}\prod _{k|n,\ k\neq 1,2}( \psi _{k}( 2-x))^{2} .
\end{align}
Since all the $ z_{n}( x)$ are monic we can replace all the polynomials of the right hand side by their monic versions.
This gives
\begin{align}\label{eq:zFact}
z_{n}( x) =\phi _{1}(x)( \phi _{2}(x))^{e_{n}}\prod _{k|n,\ k\neq 1,2}( \phi _{k}( x))^{2}
\end{align}
with $ \phi _{1}( x) =x,\ \phi _{2}( x) =x-4,\ \phi _{n}( x) =( -1)^{\varphi ( n) /2} \psi _{n}( 2-x)$.

Observing that $ 2-2\cos\left(\frac{2\pi k}{n}\right) =4\sin^{2}\left(\frac{k\pi }{n}\right)$ we get from Equation \eqref{eq:psin} by changing $ x\mapsto 2-x$
\begin{align}
\phi _{n}( x) =\prod _{\gcd( k,n) =1,\ 0< k< n/2}\left( x-4\sin^{2}\left(\frac{k\pi }{n}\right)\right) =( -1)^{\varphi ( n) /2}\frac{C_{n}( \lambda ( x))}{( \lambda ( x))^{\varphi ( n) /2}}
\end{align}
for $ n\geq 3$ with $ \lambda ( x) =\alpha ( 2-x) =\frac{2-x+\sqrt{x^{2} -4x}}{2}$. Thus $ \phi _{n}( x)$ is the minimal polynomial of $ 4\sin^{2}\left(\frac{\pi }{n}\right)$. The first terms are:
\begin{align*}
\begin{array}{ c|c c c c c c c }
n & 1 & 2 & 3 & 4 & 5 & 6 & 7\\
\hline
\phi _{n}( x) & x & -4+x & -3+x & -2+x & 5-5x+x^{2} & -1+x & -7+14x-7x^{2} +x^{3}
\end{array} \\
\ \ \ \ \ \ \ \ \ \ \ \ \ \ \ \ \ \ \ \ \ \ \ \ \ \ \ \ \
\begin{array}{ c c c }
8 & 9 & \cdots \\
\hline
2-4x+x^{2} & -3+9x-6x^{2} +x^{3} & \cdots 
\end{array} .
\end{align*}
(see \cite[entry A232633]{OEIS}).

Let 
\begin{align}
\Phi _{n}( x) &:=( \phi _{n}( x))^{2} \ \text{for} \ n\geq 3,\\
\nonumber\Phi _{1}( x) &:=x,\ \Phi _{2}( x) :=4-x=-\phi _{2}( x)
\end{align}
such that $ \phi _{1}( x)( \phi _{2}( x))^{e_{n}} =( -1)^{n-1} \Phi _{1}( x)( \Phi _{2}( x))^{e_{n}}$. Then
$ \Phi _{n}( x) \in \mathbb{Z}[ x]$ \ with $ \deg \Phi _{n} =\varphi ( n)$.

With this the proof of Theorem \ref{thm:GohWild} is a consequence of Equation \eqref{eq:zFact}. 

\section{How to calculate the $\phi_k(x)$}
\subsection{Odd index}
For odd $ m$ we know from Remark \ref{rm:ZL} that $ Z_{m}\left( x^{2}\right) =( L_{m}( x))^{2}$. This implies
\begin{align}
    L_{m}( x) =x\prod _{d|m,\ d >1} \phi _{d}\left( x^{2}\right).
\end{align}
For example, 
\begin{align*}
L_{1}( x) &=x,\ L_{3}( x) =x\left( x^{2} -3\right) =x\phi _{3}\left( x^{2}\right) ,\ L_{5}( x) =x\left( x^{4} -5x+5\right) =x\phi _{5}\left( x^{2}\right) ,\\
L_{7}( x) &=x\left( x^{6} -7x^{4} +14x^{2} -7\right) =x\phi _{7}\left( x^{2}\right) ,\ L_{9}( x) =x\left( x^{2} -3\right)\left( x^{6} -6x^{4} +9x^{2} -3\right) =x\phi _{3}\left( x^{2}\right) \phi _{9}\left( x^{2}\right) .
\end{align*}
Thus, for odd $m$ the polynomial $\phi _{m}( x)$ can be easily calculated in terms of the Lucas polynomials.

\subsection{The case indexed by powers of $2$}
Observing that $ Z_{2^{k}}( x) =\prod _{j=0}^{k} \Phi _{2^{j}}( x)$ and $ \Phi _{2^{k}}( x) =( \phi _{2^{k}}( x))^{2}$ we get for $ k\geq 2$ that
\begin{align}
    \Phi _{2^{k}}( x) =\frac{Z_{2^{k}}( x)}{Z_{2^{k-1}}( x)} =\frac{1}{( \lambda ( x))^{2^{k-1}}}\left(\frac{( \lambda ( x))^{2^{k}} -1}{( \lambda ( x))^{2^{k-1}} -1}\right)^{2} =\frac{\left( \lambda ( x))^{2^{k-1}} +1\right)^2} {( \lambda ( x))^{2^{k-1}}} =\left(( \lambda ( x))^{2^{k-2}} +\frac{1}{\lambda ( x))^{2^{k-2}}}\right)^2
    \end{align}
implies
\begin{align}\label{eq:phiRecPo2}
    \phi _{2^{k}}( x) =( \lambda ( x))^{2^{k-2}} +\frac{1}{( \lambda ( x))^{2^{k-2}}} =\left(( \lambda ( x))^{2^{k-3}} +\frac{1}{( \lambda ( x))^{2^{k-3}}}\right) -2=( \phi _{2^{k-1}}( x))^{2} -2
    \end{align}
for $ k\geq 3$. The first terms are:
$$ \begin{array}{ c|c c c c c c}
n & 1 & 2 & 3 & 4 & 5&\cdots\\
\hline
\phi _{2^{n}}( x) & x & -4+x & -2+x & 2-4x+x^{2} & 2-16x+20x^{2} -8x^{3} +x^{4}&\cdots
\end{array}.$$

The same consideration for $ L_{2^{n}}\left( z+\frac{1}{z}\right) =z^{2^{n}} +\frac{1}{z^{2^{n}}}$ together with $ L_{2}( x) =x^{2} -2$ and $ \phi _{2^{2}}\left( x^{2}\right) =x^{2} -2$ gives
\begin{align}
    \phi _{2^{n+1}}\left( x^{2}\right) =L_{2^{n}}( x)
    \end{align}
for $ n\geq 1$.
\subsection{General case}

\begin{theorem}
For odd $ m\geq 3$ and $ k\geq 2$
\begin{align}
\label{eq:phievenComp}\phi _{2m}( x) &=( -1)^{\varphi (m) /2} \phi _{m}( 4-x) =( -1)^{\varphi (m) /2}( \phi _{m} \circ \Phi _{2})( x) ,\\
\label{eq:phiPot2Comp}\phi _{2^{k} m}( x) &=\phi _{m}\left(( \phi _{2^{k}}( x))^{2}\right) =( \phi _{m} \circ \Phi _{2^{k}})( x) .
\end{align}    
\end{theorem}

\begin{proof}
We have $ Z_{2m}( x) =Z_{m}( x) Z_{m}( 4-x)$ because $ Z_{m}( 4-x) =2-L_{m}( 2-4+x) =2-L_{m}( x-2) =2+L_{m}(2- x)$
and $ Z_{2m}( x) =2-L_{2m}( 2-x) =( 2-L_{m}( 2-x))( 2+L_{m}( 2-x))$ (see Equation \eqref{eq:L1stBinomi}). 

Comparing $ Z_{2m}( x) =\prod _{d|2m} \Phi _{d}( x) =\prod _{d|m} \Phi _{d}( x) \Phi _{2d}( x)$ and 
$$ Z_{2m}( x) =Z_{m}( x) Z_{m}( 4-x) =\prod _{d|m} \Phi _{d}( x) \Phi _{d}( 4-x)$$ gives
$ \Phi _{2m}( x) =\Phi _{m}( 4-x)$ for odd $ m\geq 3$. Since $ \Phi _{m}( x) =( \phi _{m}( x))^{2}$ and $ \phi _{m}$ is monic
 of degree $ \varphi ( m) /2$ we get Equation \eqref{eq:phievenComp}.

We reproduce here the proof of Tri Nguyen \cite{TriNguyen} of Equation \eqref{eq:phiPot2Comp}. It is based on the fact that $ \phi _{n}( x)$ is the minimal polynomial of $ 4\sin^{2}\left(\frac{\pi }{n}\right)$ and has degree $ \deg \phi _{n} =\frac{\varphi(n) }{2}$. It is sufficient to show that the root $ \gamma _{k} =4\sin^{2}\left(\frac{\pi }{2^{k} m}\right)$ of $ \phi _{2^{k} m}( x)$ is also a root of $ \phi _{m}\left(( \phi _{2^{k}}( x))^{2}\right)$ as $ \phi _{2^{k} m}( x)$ is irreducible and $ \deg\left( \phi _{m} \circ ( \phi _{2^{k}}( x))^{2}\right) =\frac{\varphi ( m)}{2} 2^{k-1} =\frac{\varphi \left( 2^{k} m\right)}{2} =\deg \phi _{2^{k} m}$.

Let $ f^{( 1)}( x) =f( x) =x^{2} -2$ and $ f^{( n)}( x) =f\circ f^{( n-1)}( x)$ for $ n\geq 2$. By Equation \eqref{eq:phiRecPo2}
$ f( \phi _{2^{k-1}}( x)) =\phi _{2^{k}}( x)$, and therefore we have for $ 1\leq j< k$ 
$$ f^{( k-j)}( \phi _{2^{j}}( x)) =\phi _{2^{k}}( x).$$

Let us first show by induction that for any $ \theta \in \mathbb{R}$
\begin{align}\label{eq:itid} 
f^{( n)}\left( 4\sin^{2}( \theta ) -2\right) =2-4\sin^{2}\left( 2^{n} \theta \right) .
\end{align}
This is true for $ n=1$ since
$$f\left( 4\sin^{2}( \theta )\right) -2) =\left( 4\sin^{2}( \theta ) -2\right)^{2} -2=\left( 2\cos^{2}( \theta ) -2\sin^{2}( \theta )\right)^{2} -2=4\cos^{2}( 2\theta ) -2=2-4\sin^{2}( 2\theta ) .$$
Equation \eqref{eq:itid} follows now from
$$ f^{( n)}\left( 4\sin^{2}( \theta )\right) -2) =f\left( f^{( n-1)}\left( 4\sin^{2}( \theta ) -2\right)\right) =f\left( 2-4\sin^{2}\left( 2^{n-1} \theta \right)\right) =2-4\sin^{2}\left( 2^{n} \theta \right).$$

Now let us compute $ \phi _{2^{k}}( \gamma _{k})$. For $ k=2$ we find
$$ \phi _{4}( \gamma _{2}) =\gamma _{2} -2=4\sin^{2}\left(\frac{\pi }{4m}\right) -2=2\sin^{2}\left(\frac{\pi }{4m}\right) -2\cos^{2}\left(\frac{\pi }{4m}\right) =-2\cos\left(\frac{\pi }{2m}\right) ,$$
and for $ k >2$
\begin{align*}
\phi _{2^{k}}( \gamma _{k}) &=f^{( k-2)}( \phi _{4}( \gamma _{k})) =f^{( k-2)}( \gamma _{k} -2) =f^{( k-2)}\left( 4\sin^{2}\left(\frac{\pi }{2^{k} m}\right) -1\right) =2-4\sin^{2}\left(\frac{2^{k-2} \pi }{2^{k} m}\right)\\
&=2-4\sin^{2}\left(\frac{\pi }{4m}\right) =2\cos\left(\frac{\pi }{2m}\right).
\end{align*}

Finally, using Equation \eqref{eq:phievenComp} we get
$$ \left( \phi _{m} \circ \phi _{2^{k}}^{2}\right)( \gamma _{k}) =\phi _{m}\left( 4\cos^{2}\left(\frac{\pi }{2m}\right)\right) =( -1)^{\varphi ( m) /2} \phi _{2m}\left( 4\sin^{2}\left(\frac{\pi }{2m}\right)\right) =0.$$
\end{proof}

\begin{remark}
    The relation $ Z_{m}( \Phi _{2^{n}}( x)) =\Phi _{2^{n}}( Z_{m}( x))$ holds for $ n\geq 2$.
\end{remark}
\begin{proof}
For $ n=2$ we have $ \Phi _{4}( x) =( x-2)^{2}$. This gives
$ Z_{m}( \Phi _{4}( x)) =Z_{m}\left(( x-2)^{2}\right) =( L_{m}( x-2))^{2} =( L_{m}( 2-x))^{2} =( 2-Z_{m}( x))^{2} =\Phi _{4}( Z_{m}( x))$.    
By induction we get
\begin{align*}
Z_{m}( \Phi _{2^{n+1}}( x)) &=Z_{m}\left(\left(( \phi _{2^{n}}( x))^{2} -2\right)^{2}\right) =\left( L_{m}\left(( \phi _{2^{n}}( x))^{2} -2\right)\right)^{2} =\left( 2-Z_{m}\left(( \phi _{2^{n}}( x))^{2}\right)\right)^{2}\\
&=( 2-Z_{m}( \Phi _{2^{n}}( x)))^{2} =( 2-\Phi _{2^{n}}( Z_{m}( x)))^{2} =\left(( \phi _{2^{n}}( Z_{m}( x)))^{2} -2\right)^{2}\\
&=( \phi _{2^{n+1}}( Z_{m}( x)))^{2} =\Phi _{2^{n+1}}( Z_{m}( x)),
\end{align*}
where we have been using Remark \ref{rm:ZL} and Equation \eqref{eq:phiRecPo2}. 
\end{proof}

\section{Final remarks}

Let us finally show some connection with the Fibonacci numbers $ ( F_{n})_{n\geq 0} =( 0,1,1,2,3,5,8,13,21,\dotsc )$.
They satisfy $ F_{n} =F_{n-1} +F_{n-2}$ \ for $ n\geq 2$ with $ F_{0} =0,\ F_{1} =1$ and Binet's Formula gives
\begin{align}
    F_{n} =\frac{\rho ^{n} -\sigma ^{n}}{\rho -\sigma }
    \end{align}
with $ \rho =\frac{1+\sqrt{5}}{2}$ and $ \sigma =\frac{-1}{\rho } =\frac{1-\sqrt{5}}{2}$. It is known (see \cite{Nowicki}) that the Fibonacci numbers form what is called a divisibility sequence.

Since $ \lambda ( 5) =\frac{2-5+\sqrt{25-20}}{2} =\frac{-3+\sqrt{5}}{2} =-\left(\frac{1-\sqrt{5}}{2}\right)^2 =-\sigma ^{2}$ we get
$ ( \lambda ( 5))^{n} +\frac{1}{( \lambda ( 5))^{n}} =( -1)^{n}\left( \sigma ^{2n} +\rho ^{2n}\right) =( -1)^{n}\left( 5\left(\frac{\rho ^{n} -\sigma ^{n}}{\sqrt{5}}\right) +2( -1)^{n}\right) =5( -1)^{n} F_{n}^{2} +2,$
which implies
\begin{align}
    Z_{n}( 5) =( -1)^{n-1} 5F_{n}^{2} .
    \end{align}

Theorem \ref{thm:GohWild} gives then the factorization of the Fibonacci numbers into primitive parts (cf. \cite[entry A061446]{OEIS} and \cite{Nowicki})
$$ F_{n} =\prod _{d|n} p_{d}$$
with $p_1=1$ and $p_{n} =|\phi _{n}( 5) |$ for $n\geq 2$. We get:
$$ \begin{array}{ c|c c c c c c c c c c c c c c c c c }
n & 1 & 2 & 3 & 4 & 5 & 6 & 7 & 8 & 9 & 10 & 11 & 12 & 13 & 14 & 15 & 16 & \\
\hline
\phi _{n}( 5) & 5 & 1 & -2 & -3 & 5 & -4 & -13 & 7 & -17 & 11 & -89 & 6 & 233 & -29 & 61 & 47  & \cdots 
\end{array} .$$
For example
$$ \begin{array}{ c|c c c c c c c c c }
n & 1 & 2 & 3 & 4 & 5 & 6 & 7 & 8 & 9\\
\hline
F_{n} =\prod _{d|n} p_{d} & 1 & 1\cdot 1 & 1\cdot 2 & 1\cdot 1\cdot 3 & 1\cdot 5 & 1\cdot 1\cdot 2\cdot 4 & 1\cdot 13 & 1\cdot 1\cdot 3\cdot 7 & 1\cdot 2\cdot 17
\end{array} .$$ 
\bibliographystyle{amsplain}
\bibliography{gohwild.bib}

\end{document}